\documentclass[12pt]{article}
\usepackage{latexsym,amsfonts,amssymb,mathrsfs,float}
\usepackage{dsfont}
\usepackage{amsthm}
\usepackage[numbers]{natbib}
\usepackage{amsmath}
\usepackage{indentfirst}
\usepackage{bbm}

\usepackage{color}
\allowdisplaybreaks

\def \cal{\mathcal}

\newtheorem{thm}{Theorem}[section]

\newtheorem{lem}[thm]{Lemma}
\newtheorem{pro}[thm]{Proposition}

\newtheorem{rem}[thm]{Remark}

\numberwithin{equation}{section}

\begin{document}

\title{\bf The range of once-reinforced random walk on the half-line}

\author{Zechun Hu$^{1}$, Ting Ma$^{2,}$\footnote{Corresponding author}\ ,  Renming  Song$^{3}$   and Li Wang$^{1}$\\ \\
  {\small $^1$ College of Mathematics, Sichuan  University,
 Chengdu 610065, China}\\
 {\small zchu@scu.edu.cn; wangli0@stu.scu.edu.cn}\\ \\
  {\small $^2$School of Mathematical Sciences,
Chengdu University of Technology,  Chengdu, 610059, China}\\
{\small matingting2008@yeah.net} \\ \\
  {\small $^3$ Department of Mathematics,
University of Illinois Urbana-Champaign, Urbana, IL 61801, USA}\\
 {\small rsong@illinois.edu}}
\maketitle

\begin{abstract}

In this paper, we consider a once-reinforced random walk on the half-line, and give
the limiting behaviors of all the moments of its range.
 \end{abstract}

\smallskip


\smallskip

\noindent {\bf 2020 MR Subject Classification (2020)}\quad 60K35

\section{Introduction and main result}

The once-reinforced random walk (ORRW) was introduced in Davis \cite{DB} and it follows a very simple rule: Initially every edge carries weight $1$, and the walker traverses an edge with probability proportional to its current weight. The first time an edge is traversed, its weight is permanently reinforced
by a positive parameter,
and remains unchanged thereafter. Inspired  by the work of Pfaffelhuber and Stiefel \cite{PS}, we study the range of ORRW  on the half-line $\mathbb{Z}_{+}$ and establish asymptotic behaviors of  all the moments of the range.

Let $c>0$. An ORRW $X=(X_{n})_{n \in \mathbb{Z}_{+}}$ on $\mathbb{Z}_{+}$ with parameter $c$ is a process on $\mathbb{Z}_{+}$ such that for any $n \in \mathbb{Z}_+$,
\begin{align}
\mathbf{P}(X_{n+1} = X_{n} + 1 \mid \mathcal{F}_{n}) &= 1 - \mathbf{P}(X_{n+1} = X_{n} - 1 \mid \mathcal{F}_{n}) \nonumber \\
&=\left\{
\begin{array}{cl}
1 & \text{if } X_n = 0, \\
\frac{1}{2} & \text{if } 0 < X_{n} < R_{n}, \\
\frac{1}{1+c}& \text{if } X_{n} = R_{n},
\end{array}
\right.
\end{align}
where
\begin{equation}\label{2.2}
R_{n}:= \max\{X_{0},X_{1},\ldots, X_{n}\}.
\end{equation}
Note that $R_n+1$ is the cardinality of the range ${\cal R}_n:=\{x\in \mathbb{Z}_+: \exists m\in \{0, \dots, n\} \mbox{ such that } X_m=x\}$ of the ORRW by time $n$.

The  ORRW is closely related to
the linearly edge-reinforced random walk (LERRW). The LERRW
was introduced by Coppersmith and Diaconis \cite{CD} to model the behavior of a random walker who has a ``preference" for traversing edges he has been visited before. In this  process,  each edge initially has weight 1, and each time an edge is traversed, its weight is increased by a fixed positive constant.
The LERRW
satisfies certain partial exchangeability.
A number of results concerning the recurrence and transience of
LERRW
have been obtained (see, e.g., \cite{AO, DM, SC, SC2}). The model has  found applications in areas such as  biology (\cite{SA}) and social networks (\cite{PR}).

In contrast, the study of ORRW presents greater  challenges  due to the loss of partial
exchangeability.
Existing research on ORRW  has mainly focused on three families of graphs: tree structures (e.g., regular trees, Galton-Watson trees, and $\mathbb{Z}^{d}$-like trees),  ladder graphs (e.g., $\mathbb{Z}\times \{0,1\}$, $\mathbb{Z}\times \{0,1, \ldots, d\}$ with $d \geq 2$, $ \mathbb{Z}\times \Gamma$ ($\Gamma$ is  finite)), and integer lattices (e.g., $\mathbb{Z}^{d}$ with $d \geq 1$).

For tree structures, a rich body of work has devoted to phase transitions between  recurrence and transience. Durrett et al. \cite{DR} showed that  the ORRW is transient on regular trees for any reinforcement parameter $c>1$.  This result was extended by  Collevecchio \cite{CA1}   to supercritical Galton-Watson trees. Kious and Sidoravicius \cite{KD} further identified a phase transition on $\mathbb{Z}^{d}$-like trees. Subsequently, Collevecchio et al. \cite{CA2} fully resolved the question of transience/recurrence for ORRW on trees and showed that the critical parameter for  recurrence/transience of ORRW coincides with the so-called the branching-ruin number of the tree.
Collevecchio et al. \cite{CA3} also studied the speed of once-reinforced biased random walk (ORbRW) on trees, proving that on Galton-Watson trees without leaves, the speed is positive in the transient regime.

For ladder graphs, Sellke \cite{ST}  first investigated the ladder $\mathbb{Z} \times \{1, \ldots, d\}$  with $d \geq 2$  and established  almost-sure recurrence for small reinforcement. Later, Vervoort \cite{VM} showed recurrence for all sufficiently large reinforcement parameters. Kious et al. \cite{KD1} revisited Vervoort's unpublished work \cite{VM}   and proved   recurrence on any graph of the form $\mathbb{Z} \times \Gamma$ (with  $\Gamma$ finite) for sufficiently large reinforcement parameters. Huang et al. \cite{HX} observed that both Sellke \cite{ST} and Vervoort \cite{VM} proved that,  for the simplest ladder $\mathbb{Z} \times \{0, 1\}$,   the walk is almost surely recurrent for any reinforcement parameter
$c>1/2$. Huang et al. \cite{HX} proved that the ORRW on  $\mathbb{Z} \times \{0, 1\}$ is almost surely recurrent for
 any $c\in (0,1)$.

Sidoravicius conjectured that the ORRW  on $\mathbb{Z}^d$ is recurrent for $d\in \{1,2\}$ and undergoes a phase transition for $d\geq 3$, being recurrent when the reinforcement parameter $c$ is large and transient when it is small.   Vervoort \cite{VM} confirmed this conjecture for $d=1$. 
Recently, Elboim and Kozma \cite{ED} proved that for $d \geq 6$ and sufficiently small reinforcement parameters, the ORRW on $\mathbb{Z}^{d}$ is transient among other things.

In addition to these  recurrence/transience results, several results on the limit behavior of the ORRW have been obtained. Using Tauberian theory, Pfaffelhuber and Stiefel \cite{PS} obtained asymptotic
behaviors
for all the moments of the range of the ORRW on $\mathbb{Z}$. Specifically, the first moment of the range grows asymptotically like $\sqrt n$. This confirms the diffusive character of the range on the full line.  More recently, Collevecchio and Tarr\`{e}s \cite{CA4} extended
the analysis of  \cite{PS}
to general graphs via a change of measure formula. They derived large deviation estimates for the range and proved that on $\mathbb Z^{d}$ ($d\ge 2$)
the cardinality of the range at time $n$ is typically of order $n^{d/(d+2)}$.
However, the technique of Collevecchio and Tarr\`{e}s \cite{CA4}
 does not apply to
the setting with  boundary, such as the half-line.
Recent study on ORRWs on the  the half-line
has  centerd on limit theorems. Collevecchio et al. \cite{CA5} established a law of the iterated logarithm (LIL) for the ORRW on $\mathbb{Z}_{+}$, showing
$
\limsup_{n\to\infty}\frac{c X_{n}}{\sqrt{2n\log\log n}}=1\ \text{a.s.},
$
where $c>0$
 is a reinforcement parameter. For further results on reinforced random walks on the half-line, we refer the reader to \cite{AJ, HSW, TM}.

In this paper, motivated  by \cite{PS}, we study the range of the ORRW on the half-line, and derive asymptotic behaviors for all the moments of its range.
In this article,
the notation  $a_{n}\sim b_{n}$ means that  $a_{n}/b_{n}\stackrel{n \to \infty}{\longrightarrow }1$. We use  $\log$  to denote the logarithm with base $e$.
The main result is as follows.

\begin{thm}\label{thm-1.1}
 Let $R_{n}$, defined in \eqref{2.2}, be  the range of the ORRW on $\mathbb{Z}_{+}$ with parameter $c>0$. Then,
\begin{equation*}
\mathbb{E}\left[\left(\frac{R_{n}}{\sqrt{n}}\right)^{\ell}\right] \sim \frac{1}{2^{(\ell-2)/2}
\Gamma(\ell/2)} \cdot J_{\ell}(c),\quad \ell=1,2, \ldots,
\end{equation*}
where
\begin{equation*}
J_{\ell}(c):=2^{c} \int_{0}^{\infty} x^{\ell-1}\left(\frac{e^{x}}{e^{2x}+1}\right)^{c}dx.
\end{equation*}
In particular,
\begin{equation*}
\mathbb{E}\left[\frac{R_{n}}{\sqrt{n}}\right] \sim \sqrt{\frac{2}{\pi}} J_{1}(c) \quad \text { and } \quad \mathbb{E}\left[\frac{R_{n}^{2}}{n}\right] \sim J_{2}(c).
\end{equation*}
\end{thm}

\begin{rem}
Compared with the asymptotic behaviors of the moments of the range of the ORRW on $\mathbb{Z}$ obtained in \cite{PS}, one sees that the asymptotic orders of the moments of the range of the ORRW on
$\mathbb{Z}_+$ are the same as that on $\mathbb{Z}$,
but the  coefficients are different. In \cite{PS},
\begin{equation*}
J_{\ell}(c)=2^{2c} \int_{0}^{\infty} x^{\ell-1}\left(\frac{e^{x}}{(e^x+1)^2}\right)^{c} d x.
\end{equation*}
\end{rem}

\medskip

\begin{rem}\label{rem-1.2} The ORRW with $c=1$ is just the
  simple symmetric
 random walk on $\mathbb{Z}_{+}$  with the reflection at the origin 0. In this case, we have
\begin{equation*}
J_{1}(1) =2 \int_{0}^{\infty} \frac{e^{x}}{e^{2x}+1}dx \stackrel{y=e^{x}}{=}2 \int_{1}^{\infty}\frac{1}{y^{2}+1}dy=2 \arctan y  |_{1}^{\infty}=\frac{\pi}{2}
\end{equation*}
and
\begin{equation*}
\begin{aligned}
J_{2}(1)&=2 \int_{0}^{\infty}x\frac{e^{x}}{e^{2x}+1}dx=2 \int_{0}^{\infty} x\frac{e^{-x}}{1+e^{-2x}}dx\\
&=2 \int_{0}^{\infty}x e^{-x} \sum_{n=0}^{\infty}(-1)^{n}e^{-2nx}dx=2 \sum_{n=0}^{\infty} (-1)^{n} \int_{0}^{\infty}x e^{-(2n+1)x}dx\\
&=2\sum_{n=0}^{\infty}\frac{(-1)^{n}}{(2n+1)^{2}}=2G,
\end{aligned}
\end{equation*}
where $G$ is the Catalan constant. Hence
\begin{equation*}
\mathbb{E}\left[\frac{R_{n}}{\sqrt{n}}\right] \sim \sqrt{\frac{\pi}{2}}  \quad \text { and } \quad \mathbb{E}\left[\frac{R_{n}^{2}}{n}\right] \sim 2G.
\end{equation*}
\end{rem}

\section{Proof of Theorem \ref{thm-1.1}}\label{Proof Theorem}\setcounter{equation}{0}

Our treatment of the range process $R_n$ is an adaptation of the technique of  Pfaffelhuber and Stiefel \cite{PS}.

 Let $S_{k} := \inf \{ n: R_{n} = k \}$ denote the first time the range of the ORRW $X=(X_n)_{n\in\mathbb Z_{+}}$ on $\mathbb{Z}_{+}$ reaches $k$.  For $i=1,2,...$,
 define  $T_{i} := S_{i+1} - S_{i}$ as the time between the first instant the range  $R_{n}$ equals $i$ and the first instant it equals $i + 1$.  Note that $T_{i} = 1$ with probability $ 1/(1 + c)$.
Otherwise, the walk $X$ returns to $i-1$, and subsequently moves within $\{0,...,i\}$ until it again attains the range $i$. Denote this duration by $\tau_{i}$.
 $\tau_{i}$ is the  hitting time of  $i$ for a simple random walk with a reflecting barrier  at $0$ (i.e., the walk  moves to $1$ at the next step upon hitting 0) starting from $i-1$.
Upon reaching $i$,  the walk moves to $i+1$  with  probability $1/{(1+c)}$, thereby increasing the range.
  If instead it returns to $i-1$, and the procedure for again reaching $i$ from $i-1$  (as described above) repeats independently. The trial of advancing from $i$ to $i+1$ is thus repeated until it succeeds.
The number $Y_{i}$  of independent attempts needed to first achieve a transition from $i$ to $i+1$ follows a geometric  distribution
with  a success probability $1/(1+c)$. That is,  $\mathbb P (Y_i=j)=(\frac c{c+1})^{j}\frac 1{c+1}$ for  $j=0,1,2,...$.  In total, this gives
\begin{equation}\label{2.3}
S_{k}=1+\sum_{i=1}^{k-1}T_{i}, \quad k=1,2,\ldots \quad  \text{with}  \quad T_{i} = 1 + \sum_{j=1}^{Y_{i}} (1 + \tau_{i}^j), \quad i = 1, 2, \ldots,
\end{equation}
where we define the empty sum to be $0$. Here, $\tau_{i}^k$, $k = 1, 2, \ldots$ are  independent
copies of $\tau_i$ and are also independent of the sequence $ Y_i, i =1,2,\ldots$.

\subsection{Some preliminaries}

Before proving Theorem \ref{thm-1.1}, we first make some preparations.
We begin by recalling the classical Tauberian result by Hardy and Littlewood
\cite[Chapter I. Theorem 7.4]{T}
(see also  \cite[Proposition 12.5.2]{LGF}).

\begin{thm}\label{thm-3.1}
Let $a_{1}, a_{2}, \ldots \geq 0$ such that $\sum_{n=1}^{\infty} a_{n} x^{n}$ converges for $|x|<1$. Suppose that for some $\alpha, A \geq 0$,
\begin{equation*}
\sum_{n=1}^{\infty} a_{n} x^{n} \stackrel{x \uparrow 1}{\sim} \frac{A}{(1-x)^{\alpha}}.
\end{equation*}
Then,
\begin{equation*}
\sum_{k=1}^{n} a_{k} \stackrel{n \rightarrow \infty}{\sim} \frac{A}{\Gamma(\alpha+1)} n^{\alpha}.
\end{equation*}
Moreover, if $\alpha>1$, and $n \mapsto a_{n}$ is nondecreasing,
\begin{equation}\label{3.2-a}
a_{n} \sim \frac{A \alpha}{\Gamma(\alpha+1)} n^{\alpha-1}.
\end{equation}
\end{thm}

The following lemma gives the generating function of the first hitting time in a simple symmetric random walk with a reflecting barrier. Its proof adapts the classical “ruin problem” method (see e.g., \cite[Chapter XIV.4]{FW}).

\begin{lem}\label{lem-3.2}
Let $a, b, x \in \mathbb{Z}$ with $a \geq 0$ and $-a \leq x\le b$.
Consider a simple symmetric random walk $Z=(Z_{n})_{n \in \mathbb{Z}_{+}}$ starting at $x$
on the state space $\{-a,-a+1,...\}$ with a reflecting barrier at $-a$, defined by the rule: if  $Z_n=-a$, then  $Z_{n+1}=-a+1$.
Let $T_{b}:=\inf \left\{n\ge0: Z_{n}=b \right\} $ denote the first hitting time of $b$. Then for any $s \in (0,1)$
\begin{equation}\label{3.0}
\mathbb{E}_{x}\left[s^{T_{b}}\right]=\frac{r_s^{x}+r_s^{-(2a+x)}}{r_s^{b}+r_s^{-(2a+b)}},
\end{equation}
where
\begin{equation}\label{3.1}
r_s:=\frac{1+\sqrt{1-s^{2}}}{s}.
\end{equation}
\end{lem}

\begin{proof}
Let
\[
U_{x}(s):=\mathbb{E}_{x} [s^{T_{b}} ]=\sum_{n=0}^\infty \mathbb P_x(T_b=n)s^n.
\]
Then for  $-a < x < b$,
\begin{equation}\label{3.4}
U_{x}(s)=\frac{s}{2}U_{x+1}(s)+\frac{s}{2}U_{x-1}(s)
\end{equation}
and
\begin{equation}\label{3.5}
U_{-a}(s)=sU_{-a+1}(s), \quad U_{b}(s)=1.
\end{equation}
We now use the method of \cite[Chapter XIV.4]{FW} to solve the  difference equation \eqref{3.4}.
We first look for particular solutions of the form $U_{x}(s)=r ^{x}(s)$.
Substituting this expression into \eqref{3.4}, we get
that $r (s)$ must satisfy the quadratic equation
\begin{equation}\label{quadratic}
r(s)=\frac{s}{2}r^{2}(s)+\frac{s}{2},
\end{equation}
which has  two roots
\begin{equation*}
r_{1}(s)=\frac{1+\sqrt{1-s^{2}}}{s}, \quad r_{2}(s)=\frac{1-\sqrt{1-s^{2}}}{s}, \quad s \in (0,1),
\end{equation*}
satisfying $r_{1}(s) \cdot r_{2}(s)=1$. It is easy to check that
 \begin{equation}\label{quadratic-2}
1-sr_1(s)={sr_1^{-1}(s)-1}.
\end{equation}
The same identity also holds for $r_2(s)$.

According to \cite[Chapter XIV.2]{FW},
the general solution can be written  as
\begin{equation}\label{3.6}
U_{x}(s)=A(s)r_{1}^{x}(s)+B(s)r_{2}^{x}(s),
\end{equation}
where $A(s)$ and $B(s)$ are
functions of $s$. The boundary conditions \eqref{3.5} yield
\begin{equation}\label{3.7-1}
 A(s)r_{1}^{-a}(s) (1-sr_1(s))= B(s)r_{2}^{-a}(s)(sr_2(s)-1),
 \end{equation}
and
 \begin{equation}\label{3.8-1}
 A(s)r_{1}^{b}(s) + B(s)r_{2}^{b}(s)=1.
 \end{equation}
Using the  relation  $ r_{2}(s)=1/r_1(s)$ and \eqref{quadratic-2}, equation \eqref{3.7-1} simplifies to
\begin{equation}\label{3.7}
 A(s)  = B(s)r_{1}^{2a}(s).
 \end{equation}
Substituting $ r_{2}(s)=1/r_1(s)$ into   \eqref{3.8-1}  gives
 \begin{equation}\label{3.8}
 A(s)r_{1}^{b}(s) + B(s)r_{1}^{-b}(s)=1.
 \end{equation}
Solving the linear system \eqref{3.7} and \eqref{3.8} for $A(s)$ and $B(s)$, we get
 \begin{equation}\label{3.9}
 A(s)=\frac{1}{r_{1}^{b}(s)+r_{1}^{-(2a+b)}(s)}, \quad B(s)=\frac{1}{r_{1}^{2a+b}(s)+r_{1}^{-b}(s)}.
 \end{equation}
 Finally,  inserting  \eqref{3.9} into \eqref{3.6} and replacing $r_2(s)$ by $1/r_1(s)$, we easily obtain \eqref{3.0} .
 \end{proof}

\subsection{Proof of Theorem \ref{thm-1.1}}

Following \cite[Lemma 3.2]{PS},  we begin with the generating functions of $\tau_{i}$, $T_{i}$, and $S_{k}$.
Recall that,  for $s\in(0,1)$, $r_s$ is defined in \eqref{3.1}.
 Let
\begin{equation}\label{3.2}
g_{x}(s):=\frac{r_s^{x-1}+r_s^{1-x}}{r_s^{x}+r_s^{-x}}, \quad G_{x}(s):=\frac{s}{1+c-cs g_{x}(s)}{\color{blue} .}
\end{equation}

\begin{lem}\label{lem-3.3}
Let $\tau_{i}$ and $T_{i}$  be defined in \eqref{2.3}, $g_x(\cdot)$ and $G_x(\cdot)$ be defined in \eqref{3.2}. Then
for $i=1,2,3,\ldots$, and $s\in (0,1)$,
\begin{align}\label{lem-3.3-a}
\mathbb{E}\left[s^{\tau_{i}}\right]=g_{i}(s), \quad \mathbb{E}\left[s^{T_{i}}\right]=G_{i}(s).
\end{align}
Moreover, the generating function of $S_{k}$ is given by
\begin{equation}
\mathbb{E}\left[s^{S_{k}}\right]=s \prod_{i=1}^{k-1} G_{i}(s), \quad k=1,2,
\ldots.
\end{equation}
\end{lem}

\begin {proof}
Applying Lemma \ref{lem-3.2} with $a=0$,
 $x=i-1$
and $b=i$, we immediately get the first equality in
\eqref{lem-3.3-a}.

For the generating function of $T_{i}$, recall that  $Y_{i} \sim \operatorname{geo}(1 /(1+c))$ is
 the number of failures before the first success.  Its probability  generating function is
\begin{equation*}
\mathbb E\left[s^{Y_i}\right]= \frac{1}{c+1} \sum_{k=0}^{\infty}\left(\frac{c s}{c+1}\right)^{k}=\frac{1}{1+c-c s}.
\end{equation*}
Using the independence of $\{Y_i\}$ with the independent family $\{\tau^j_i\}$, we get
\begin{equation*}
\mathbb{E}\left[s^{T_{i}}\right]=s \mathbb{E}\left[\mathbb{E}\left[\prod_{j=1}^{Y_{i}} s^{1+\tau_{i}^{j}} \mid Y_{i}\right]\right]=s\mathbb{E}\left[\left(s g_{i}(s)\right)^{Y_{i}}\right]=\frac{s}{1+c-csg_{i}(s)} .
\end{equation*}

Finally, using  \eqref{2.3} we get
\begin{equation*}
\mathbb{E}\left [s^{S_{k}}\right ]=s\mathbb{E}\left[\prod_{i=1}^{k-1}s^{T_i}\right]=s\prod_{i=1}^{k-1}\mathbb{E}\left[s^{T_i}\right] =s\prod_{i=1}^{k-1}G_i(s).
\end{equation*}
This completes the proof.
\end{proof}

 Using the generating function of $S_k$, one can easily get the following result, see \cite[Lemma 3.3]{PS}. Recall $G_{x}(s)$  is defined in  \eqref{3.2}.
\begin{lem}\label{lem-3.4}
For $s \in(0,1)$ and $ \ell=0,1,2, \ldots$, it holds
that
\begin{equation}\label{3.15-a}
H_{\ell}(s):=\sum_{n=1}^{\infty} s^{n} \mathbb{E}\left[R_{n} \cdots\left(R_{n}+\ell\right)\right]=\frac{\ell+1}{1-s} \sum_{k=1}^{\infty} k \cdots(k+\ell-1) s \prod_{i=1}^{k-1} G_{i}(s).
\end{equation}
\end{lem}

The following proposition gives the asymptotic behavior of $H_{\ell}$ and its proof is deferred to Section \ref{sec4}.

\begin{pro}\label{pro-3.5}
Let $H_\ell$ for $\ell = 0, 1, 2, \ldots$ be as defined in Lemma \ref{lem-3.4}. Then,
\begin{equation*}
H_\ell(s) \stackrel{s \uparrow 1}{\sim} K_{\ell}(1-s)^{-(\ell+3)/2} ,
\end{equation*}
where
\begin{equation}\label{3.15}
K_{\ell}=\frac{\ell+1}{2^{(\ell+1)/2}}2^{c}\int_{0}^{\infty}x^{\ell}\left(\frac{e^{x}}{e^{2x}+1}\right)^{c}dx.
\end{equation}
\end{pro}

\smallskip

\begin{proof}[Proof of Theorem \ref{thm-1.1}]

Let $a_n := \mathbb{E}[R_n \cdots (R_n + \ell)]$. It follows from Proposition \ref{pro-3.5} that
\begin{equation*}
H_\ell(s) = \sum_{n=1}^\infty a_n s^n \stackrel{s \uparrow 1}{\sim} \frac{K_\ell}{(1-s)^{(\ell+3)/2}},
\end{equation*}
Thus, Theorem \ref{thm-3.1} holds with $A=K_{\ell}$ and $\alpha = (\ell+3)/2$. In particular, \eqref{3.2-a} implies that
\begin{equation*}
\mathbb{E}[R_n \cdots (R_n + \ell)]\sim
\frac{K_{\ell}\cdot (\ell +3)/2}{\Gamma((\ell+5)/2)}n^{(\ell+1)/2}.
\end{equation*}
Combining \eqref{3.15} with the identity  $\Gamma(x+1)=x\Gamma(x)$, the coefficient
can be simplified  to
\[
\frac{K_{\ell}\cdot (\ell +3)/2}{\Gamma((\ell+5)/2)}
=\frac{2^{c}}{2^{(\ell-1)/2}\Gamma((\ell+1)/2)} \int_{0}^{\infty}x^{\ell }\left( \frac{e^{x}}{e^{2x}+1} \right)^{c}dx.
\]
Consequently,  for  every integer  $\ell=0, 1, 2,\ldots$,
\[
\mathbb{E}[R_{n}^{\ell+1}] \sim
\frac{2^{c}}{2^{(\ell-1)/2}\Gamma((\ell+1)/2)} \int_{0}^{\infty}x^{\ell }\left( \frac{e^{x}}{e^{2x}+1} \right)^{c}dx  \cdot n^{(\ell+1)/2}.
\]
The proof is complete.
\end{proof}

\section{Proof of Proposition \ref{pro-3.5}}   \label{sec4}

The proof is similar to that of \cite[Proposition 3.4]{PS}. Recall $r_s$ is defined in \eqref{3.1}.
Let $d_{s}=\log r_s$.
We can then rewrite  \eqref{3.2} as
\begin{align}\label{4.2}
g_{x}(s)=
\frac{e^{(x-1)\log r_s}+e^{(1-x)\log r_s}}{e^{x\log r_s}+e^{-x\log r_s}}
=\frac{\cosh(x d_{s} -d_{s})}{\cosh(xd_{s})}
=\cosh d_{s}-\tanh(xd_{s})\sinh d_{s},
\end{align}
where the last equality follows from the identity $\cosh (u+v)=\cosh u \cosh v+\sinh u \sinh v$.
By the definition of $d_s$,  we have
\begin{equation*}\label{4.3}
\cosh d_{s}=\frac{r_s+r^{-1}_s}2=\frac 1{s}.
\end{equation*}
Using the identity $\cosh^{2}(d_{s})-\sinh^{2}(d_{s})=1$,   we get
\begin{equation*}\label{4.5}
\sinh d_{s}=\sqrt{\cosh^{2} d_{s}-1}=\sqrt{1-s^{2}}/s.
\end{equation*}
 Define
\begin{equation*}\label{4.1}
 \phi (x) :=\frac{e^{x}-1}{e^{x}+1},
 \end{equation*}
then we have
\begin{equation*}\label{4.4}
\tanh(xd_{s})=\frac{e^{xd_{s}}-e^{-xd_{s}}}{e^{xd_{s}}+e^{-xd_{s}}}=\frac{e^{2xd_{s}}-1}{e^{2xd_{s}}+1}=\phi(2xd_{s}).
\end{equation*}
Substituting the above equalities into \eqref{4.2} yields
\begin{equation}\label{4.6}
g_{x}(s)=s^{-1}\big(1-\phi(2x d_{s})\sqrt{1-s^{2}}\big).
\end{equation}
Now inserting \eqref{4.6} into \eqref{3.2} gives
\begin{equation}\label{4.7}
G_{x}(s)=\frac{s}{1+c\phi(2xd_{s})\sqrt{1-s^{2}}}.
\end{equation}

Set $ t := t(s) := 1 - s $. By the definition of $d_{1-t}$, we have
\begin{align*}
f_{t} := d_{1-t}
&= \log\left(\frac{1}{1-t}\left(1 + \sqrt{t(2-t)}\right)\right) \nonumber \\
&= -\log(1-t) + \log\left(1 + \sqrt{t(2-t)}\right) .
\end{align*}
 Applying  the  second-order Taylor expansion of $\log(1+x)$ around 0, we obtain
\begin{align}\label{4.8}
f_{t}
&= t + \sqrt{t(2-t)} - \frac{t}{2}(2-t) + O(t^{3/2}) \nonumber \\
&=\sqrt{2t}+O(t^{3/2}) \quad \text{as} \quad t \to 0.
\end{align}
 We now rewrite \eqref{4.6} and \eqref{4.7} as functions of $t$. By \eqref{4.8}, we have
\begin{align*}
g_{x}(1-t)&=\frac{1}{1-t}\left(1-\phi(2xf_{t})\sqrt{t(2-t)}\right)\\
&=(1+O(t))\left(1-\phi(2\sqrt{2t}x)\sqrt{2t}(1+O(t))\right)\\
&=1-\phi(2\sqrt{2t}x)\sqrt{2t}+O(t),
\end{align*}
and, replacing $\frac 1{1+x}$ by its first-order Taylor expansion around 0, we get
\begin{align*}\label{4.9}
G_{x}(1-t)
&=\frac{1-t}{1+c\phi(2xf_{t})\sqrt{ t(2-t)} }\nonumber \\
&=(1-t)\left(1-c\phi(2xf_{t})\sqrt{2t}(1+O(t))+O(t)\right) \nonumber \\
&=1-c\phi(2xf_{t})\sqrt{2t}+O(t)\nonumber\\
&=1-c\phi(2\sqrt{2t}x)\sqrt{2t}+O(t).
\end{align*}
Consequently,
\begin{equation*}\label{4.12} \begin{aligned}
\prod_{i=1}^{k-1}G_{i}(1-t)& =\exp\left\{\sum_{i=1}^{k-1}\log G_{i}(1-t)\right\}  \\
&=\exp\left\{\sum_{i=1}^{k-1}\log\left(1-c\phi(2\sqrt{2t}i)\sqrt{2t}+O(t)\right) \right\}   \\
&=\exp\left\{\sum_{i=1}^{k-1}-c\phi(2\sqrt{2t}i)\sqrt{2t}+O(t) \right\}.
\end{aligned}\end{equation*}
We
now deal with the series in \eqref{3.15-a}:
\begin{align}\label{4.10}
H_{\ell}(1-t)t^{\frac{\ell+3}{2}}
&= \frac{\ell+1}{2^{(\ell+1)/2}} \sum_{k=1}^{\infty}\sqrt{2t} \left(\prod_{j=0}^{\ell-1}\sqrt{2t}(k+j)\right)(1-t)  \exp\left(\sum_{i=1}^{k-1}-c\phi(2\sqrt{2t}i)\sqrt{2t}+O(t)\right)\nonumber\\
&=\frac{\ell+1}{2^{(\ell+1)/2}}\int_{0}^{\infty}x^{\ell}  \exp\left( -c \int_{0}^{x}\phi(2y)dy\right)dx\cdot(1+o(t)),
\end{align}
where we regarded the two sums as the Riemann sums of the integrals in the last line above (with $\sqrt{2t}\approx dx, dy$).
Using a simple change of variables, we get
\begin{align*}
\int_{0}^{x}\phi(2y)dy&=\int_{0}^{x}\frac{e^{2y}-1}{e^{2y}+1}dy=\int_{1}^{e^{2x}}\frac{z-1}{z+1}\cdot \frac{1}{2z}dz=\frac{1}{2}\int_{1}^{e^{2x}}\left( \frac{2}{z+1}-\frac{1}{z}\right)dz\\
&=\log(z+1)|_{1}^{e^{2x}}-\frac{1}{2}\log z |_{1}^{e^{2x}}=\log(e^{2x}+1)-\log2-x.
\end{align*}
Hence,
\begin{align*}
\exp\left(-c \int_{0}^{x}\phi(2y)dy\right)=2^{c}\left(\frac{e^{x}}{e^{2x}+1}\right)^{c}.
\end{align*}
Substituting this expression into  \eqref{4.10}  completes the proof.
\hfill\fbox

\smallskip

\noindent {\bf\large Acknowledgments}\  This work was supported by the National Natural Science Foundation of China (Grant Nos. 12171335, 12471139) and the Simons Foundation (\#960480).

\end{document}